\documentclass[11pt,a4paper]{amsart} \usepackage[utf8]{inputenc}
\usepackage[T1]{fontenc} \usepackage{fixltx2e, graphicx, longtable,
  float, wrapfig, soul, textcomp, marvosym, wasysym, latexsym,
  hyperref, bbm, fancybox,fancyvrb,mathtools,
  cite,amsmath,amssymb,amsthm,color,subfig,pdfpages, summerstylesheet}
\usepackage[usenames,dvipsnames]{xcolor} \usepackage[alphabetic,
abbrev]{amsrefs} \usepackage{listings} \usepackage{tabularx}
\newtheorem*{thm*}{Theorem} \newtheorem{thm}{Theorem}[section]
 \newtheorem{lem}[thm]{Lemma}
\newtheorem{prop}[thm]{Proposition} \theoremstyle{definition}
\newtheorem{defn}[thm]{Definition} \theoremstyle{remark}

\numberwithin{equation}{section}



\title[Sequences generated by normal numbers]{Distribution of
  sequences generated by certain simply-constructed normal numbers}
\author{Demi Allen}

\address[D.~Allen]{Department of Mathematics\\
  University of York\\
  York, UK \\
  YO10 5DD} \email[D.~Allen]{dda505@york.ac.uk}

\author{Sky Brewer}

\address[S.~Brewer]{Department of Mathematics\\
  University of York\\
  York, UK \\
  YO10 5DD}

\email[S.~Brewer]{jaco.brewer@gmail.com}

\date{October 30, 2015}

\begin{document}
\begin{abstract}
  \noindent In 1949 Wall showed that $x = 0.d_1d_2d_3 \dots$ is normal if
  and only if $(0.d_nd_{n+1}d_{n+2} \dots)_n$ is a uniformly
  distributed sequence. In this article, we consider sequences which
  are slight variants on this. In particular, we show that certain
  normal numbers of the form $0.a_na_{n+1}a_{n+2} \dots$, where $a_n$
  is a sequence of positive integers, give
  rise in a rather natural way to sequences which are not uniformly
  distributed. Motivated by a result of Davenport and Erd\H{o}s we
  also show that for a non-constant integer polynomial the sequence
  $(0.f(n)f(n+1)f(n+2) \dots)_n$ is not uniformly distributed.
\end{abstract}

\maketitle
\section{Introduction and Statement of Results}

A number, $\alpha$, is said to be normal to the base $b$ if the
frequencies of strings of digits in the $b$-adic expansion are as
would be expected if the digits were completely random. In his 1933
paper, \cite{Champernowneref}, Champernowne exhibited a selection of
numbers normal to the base $10$ with simple constructions. Most
notable of these was the so-called Champernowne's number, namely the
number $0.1234567891011121314\dots$ constructed by concatenating all
of the natural numbers in (ascending) order after the decimal point -
throughout we will denote this number by $\theta$. In 1949 it was
proved by Wall in his PhD thesis \cite{wall1949} that a real number
$\alpha$ is normal to the base $b$ if and only if the sequence
$(b^n \alpha)_n$ is uniformly distributed modulo 1. Thus we know that
the sequence $(10^n \theta)_n$ is uniformly distributed modulo $1$. A
natural question therefore is: what about the sequence
$(x_n)_n = 0.(n)(n+1)(n+2)(n+3)\dots$ where the $n$th term is
essentially constructed by taking $\theta$ but starting from the
natural number $n$ after the decimal point? For example the 20th term
of the sequence $\{ 10^n \theta\}$ would be $0.516171819 \dots$
whereas the 20th term of the sequence which we are now concerned with
is $x_{20} = 0.202122232425 \dots$. We ask the following, is this
sequence uniformly distributed modulo $1$ over a suitable subinterval
of the unit interval?

In this paper we answer this question, in the negative, and consider
various related questions. In particular, Davenport and Erd\H os
showed in \cite{DavenportErdosref} that: given a polynomial
$p:\mathbb N \rightarrow \mathbb N$ of degree $\geq 1$ the number
$0.f(1)f(2)f(3) \dots$ is normal. We consider also the distribution of
the sequence $x_n = 0.f(n)f(n+1)f(n+2)...$.

Before stating our first result we introduce some necessary
terminology and notation which will be used throughout.

For positive real numbers $x$ we will denote by $\lfloor x \rfloor$
the integer part of $x$ and by $\{x\}$ the fractional part of $x$. For
a sequence of real numbers $(x_n)_n$ and $E \subseteq [0,1]$ we will
denote by $A(E;N;(x_n)_n)$ the number of $x_n$ satisfying both
$\{x_n\} \in E$ and $1 \leq n \leq N$.

Returning to the sequence $(x_n)_n$ corresponding to Champernowne's
number defined above we note that it has no values in the interval
$[0,0.1]$ but that it is dense in the interval $[0.1,1)$. So we
consider uniform distribution over such an interval using the
following definition, which is based on Definition 1.1 given in
\cite[Chapter 1]{KUIPERS1974}.

\begin{defn} \label{u.d. def} We will say that a sequence of real
  numbers $(x_n)_n$ is \itshape uniformly distributed modulo 1 over
  $[\alpha, \beta) \subseteq [0,1)$ \normalfont, which we shall
  henceforth abbreviate to u.d. mod $1$ over $[\alpha, \beta)$, if for
  any pair of real numbers $\alpha \leq a < b \leq \beta$ we have
  \begin{align}
    \lim_{N \to \infty}{\frac{A([a,b); N; (x_n)_n)}{N}} &
                                                          =\frac{b-a}{\beta- \alpha}.
                                                          \label{u.d. mod 1 condition}
  \end{align}
\end{defn}

Our first result is inspired by the normality of Champernowne's number and Wall's result.

\begin{thm} \label{Champernowne number theorem} The sequence $(x_n)_n$
  of real numbers defined by $x_n = 0.(n)(n+1)(n+2)(n+3)\dots$, where
  the $n$th term is formed by concatenating all of the natural numbers
  in order from $n$ onwards, is not u.d. mod $1$ over $[0.1, 1)$.
\end{thm}

One can prove this quite easily by observing that, for any natural
number $J$, upon reaching the term $x_{10^J}$ the next $10^J$ terms
will begin with a 1 immediately after the decimal point. That is; for
each $J \in \mathbb{N}$ at least half of the terms up to the term
$x_{2 \times 10^J}$ begin with a first decimal digit 1. More
precisely, for $J \in \mathbb{N}$;
\[
  \frac{A([0.1, 0.2); 2 \times 10^J; (x_n)_n)}{2 \times 10^J} =
  \frac{\#\{ x_n \in [0.1,0.2): n \leq 2 \times 10^J\} }{2 \times
    10^J} \geq \frac{1}{2}.
\]
Comparing this with Definition \ref{u.d. def} the result of Theorem
\ref{Champernowne number theorem} follows. The point is that there
are too many terms of the sequence in the interval $[0.1, 0.2)$
infinitely often. \\

As well as proving the normality of $\theta$ in \cite{Champernowneref}
Champernowne also highlights a few other very natural constructions of
decimals which turn out to be normal - for example the number
$0\cdot46891012141516182021\dots$ formed by concatenating all of the
composite numbers in ascending order. The motivation for our next
result is one such construction considered by Champernowne in
\cite{Champernowneref}, namely:

\begin{thm*}[Champernowne] If $k$ is
  any positive number and $a_n$ denotes the integral part of $kn$,
  then the decimal $0\cdot a_1a_2\dots a_n \dots$ is normal in the
  scale of ten.
\end{thm*}

We remark that Champernowne does not provide an explicit proof of this
statement (or indeed of the normality of
$0\cdot46891012141516182021\dots$) in \cite{Champernowneref}. However,
this can be verified by Copeland and Erd\H{o}s' result in
\cite{CopelandErdosref}. So, taking $k \in \mathbb{N}$ in Champernowne's Theorem stated above we obtain the normal number
$0 \cdot k(2k)(3k)(4k) \dots$. However, along the same lines as
Theorem \ref{Champernowne number theorem}, when we ask the analogous
question here to the one posed in the introduction regarding
Champernowne's number, we obtain the following result.

\begin{thm} \label{general kn theorem} Let $k \in \Nat$ be arbitrary.
  Then, the sequence $(x_n)_n$ defined by
  \[
    x_n = 0.(kn)(k(n+1))(k(n+2))(k(n+3))\dots
  \]
  is not u.d. mod 1 over $[0.1,1)$.
\end{thm}


Finally, motivated by the result of Davenport and Erd\H{o}s in
\cite{DavenportErdosref} we establish the following theorem.

\begin{thm} \label{Davenport-Erdos theorem} Let
  $f(n) = c_dn^d + c_{d-1}n^{d-1} + \dots + c_1n + c_0$ be a
  non-constant polynomial with real coefficients such that for
  $n \in \Nat$ we have $f(n) \in \Nat$. Define a sequence by
  $x_n = 0.f(n)f(n+1)f(n+2)f(n+3)\dots$. Then, the sequence $(x_n)_n$
  is not u.d. mod 1 over $[0.1,1)$.
\end{thm}

\section{Proofs}

We begin this section with a lemma which is the key to establishing Theorem \ref{general kn theorem}.

\begin{lem} \label{A estimate lemma} Let $a_n = kn$ and define the
  sequence $(x_n)_n$ by $x_n = 0 \cdot a_n a_{n+1}a_{n+2}\dots$ then
  \[ A\left([0.1,0.2); \left\lfloor\frac{2 \times
          10^{J}}{k}\right\rfloor; (x_n)_n\right) =
    \sum_{i=0}^{J}{\left(\left\lfloor\frac{10^{i}}{k}\right\rfloor +
        O(1)\right)}.\]
\end{lem}

\begin{proof}[Proof of Lemma \ref{A estimate lemma}]
  We observe that
  $A\left([0.1,0.2),\lfloor \frac{2 \times 10^J}{k} \rfloor; (x_n)_n
  \right)$
  counts the number of terms $a_n$ with leading digit 1 and
  $n \leq \lfloor \frac{2 \times 10^J}{k} \rfloor$. That is, the
  number of terms $a_n$ satisfying $10^j \leq a_n < 2 \times 10^j$ for
  some $j \in \mathbb{N}$ with $0 \leq j \leq J$. So, we may write
  \begin{align*}
    A\left([0.1,0.2),\left\lfloor \frac{2 \times 10^J}{k} \right\rfloor; (x_n)_n \right) 
    &= \sum_{j=0}^{J}{\#\{n: 10^j \leq a_n < 2 \times 10^j \}} \\
    &= \sum_{j=0}^{J}{\left(\left\lfloor \frac{10^j}{k} \right\rfloor + O(1) \right)},
  \end{align*}
  which is the desired result.
\end{proof}

We may now proceed to prove Theorem \ref{general kn theorem}.

\begin{proof}[Proof of Theorem \ref{general kn theorem}]

We begin by recalling the fact that, given a bounded sequence of real numbers $(x_n)_n$, we have
  \[\limsup{n \to \infty}{x_n} \geq \limsup{k \to
      \infty}{x_{n_k}} \]
where $(x_{n_k})_{n_k}$ is any subsequence of $x_n$. 

  Using this fact in conjunction with Lemma \ref{A estimate lemma} we will show that
  \[
    \limsup{N \to \infty}{A([0.1,0.2); N; (x_n)_n) \geq \frac{5}{9}}.
  \]
  This suffices to show that the sequence $(x_n)_n$ is not u.d. mod 1
  over $[0.1, 1)$ since, if it were, we would have
  \[
    \limsup{N \to \infty}{\frac{A([0.1,0.2); N; (x_n)_n)}{N}} =
    \lim_{N \to \infty}{\frac{A([0.1,0.2); N; (x_n)_n)}{N}} =
    \frac{0.2 - 0.1}{1 - 0.1} = \frac{1}{9} < \frac{5}{9}.
  \]
  We will consider the value of $\frac{A([0.1,0.2); N; (x_n)_n)}{N}$
  evaluated at each of the points of the subsequence $(n_j)_j$ of the
  natural numbers defined by
  \[
    n_j = \left\lfloor\frac{2 \times 10^{j}}{k}\right\rfloor \text{
      for all } j> \frac{\log{k} - \log{2}}{\log{10}}
  \]
  (the condition imposed on $j$ ensures that
  $\left\lfloor\frac{2 \times 10^{j}}{k}\right\rfloor \geq 1$).
  
  By Lemma \ref{A estimate lemma} we have
  \[
    A\left([0.1,0.2); n_j; (x_n)_n\right) =
    \suml{i=0}{j}{\left(\left\lfloor\frac{10^i}{k}\right\rfloor +
        O(1)\right)}.
  \]
  From which it follows that
  \begin{align*}
    \limsup{j \to \infty}{\frac{A\left([0.1,0.2);
    n_j;(x_n)_n\right)}{n_j}} &= \limsup{j \to\infty}{\frac{
                                \suml{i=0}{j}{\left(\left\lfloor\frac{10^i}{k}\right\rfloor
                                +O(1)\right)}}{\left\lfloor\frac{2 \times
                                10^{j}}{k}\right\rfloor}} \\
                              &\geq \limsup{j \to \infty}{\frac{k}{2}
                                \frac{\suml{i=0}{j}{\frac{10^i}{k}} +O(j)}{10^{j}}} \\
                              & = \frac{k}{2} \cdot \frac{1}{k}
                                \cdot\suml{i=0}{\infty}{\frac{1}{10^i}} = \frac{5}{9},
  \end{align*}
  where the last equality is obtained by observing that
  $\suml{i=0}{\infty}{\frac{1}{10^i}}$ is a geometric series.

  Since we have now shown that there is a subsequence of the natural
  numbers $(n_j)_j$ for which
  \[
    \limsup{j \to \infty}{\frac{A\left([0.1,0.2); n_j;
          (x_n)_n\right)}{n_j}} \geq \frac{5}{9}
  \]
  it follows that
  \[
    \limsup{N \to \infty}\frac{A([0.1,0.2); N; (x_n)_n)}{N} \geq
    \frac{5}{9}.
  \]
  This concludes the proof of Theorem \ref{general kn theorem}.
\end{proof}

The main tool we use to prove Theorem \ref{Davenport-Erdos theorem} is
a generalisation of Lemma \ref{A estimate lemma}.

\begin{lem} \label{Davenport-Erdos A lemma} Let
  $f(n) = c_dn^d + c_{d-1}n^{d-1} + \dots + c_1n + c_0$ be a
  polynomial with real coefficients and of degree $\geq 1$ such that
  for $n \in \Nat$ we have $f(n) \in \Nat$. Define a sequence by
  $x_n = 0.f(n)f(n+1)f(n+2)f(n+3)\dots$. Then for $J \in \Nat$ we have
  \[ A\left([0.1,0.2); f^{-1}(2 \times 10^J); (x_n)_n\right) =
    \left(\frac{2^{1/d}-1}{c_d^{1/d}}\right)
    \sum_{i=1}^{J}{(10^{i/d})} + O(J). \]
\end{lem}

In order to prove Lemma \ref{Davenport-Erdos A lemma}, and
subsequently Theorem \ref{Davenport-Erdos theorem}, we require the
following observation:

\begin{lem} \label{order lemma} Let
  \begin{align}
    f(n) = c_dn^d + c_{d-1}n^{d-1} + \dots + c_1n + c_0 \label{polynomial} 
  \end{align}
  be a polynomial of degree $d$ and let $g(m)$ be its eventually
  monotonically increasing inverse (i.e. $f(g(m)) = m$). Then
  $g(m) = m^{1/d} c_d^{-1/d} + \varepsilon(m)$ and as
  $m \rightarrow \infty$ we have $\varepsilon(m) = O(1)$.
\end{lem}

\begin{proof}
  First, we substitute $n = m^{1/d} c_d^{-1/d} + \varepsilon$ into
  (\ref{polynomial}) to get
  \[
    f(m^{1/d} c_d^{-1/d} + \varepsilon) = c_d(m^{1/d} c_d^{-1/d} +
    \varepsilon)^d + c_{d-1}(m^{1/d} c_d^{-1/d} + \varepsilon)^{d-1} +
    \dots + c_0.
  \]
  Using a combination of the Binomial theorem and Taylor
  expansions we may establish that
  \[ \varepsilon = -\frac{c_{d-1}}{dc^{1-2/d}} + O(m^{-2/d}).\]
\end{proof}

With this in mind, we may now proceed to prove Lemma
\ref{Davenport-Erdos A lemma}.

\begin{proof}[Proof of Lemma \ref{Davenport-Erdos A lemma}]
  The idea behind this proof is the same as that used to establish Lemma
  \ref{A estimate lemma}.

  Let $ f(n) = c_dn^d + c_{d-1}n^{d-1} + \dots + c_1n + c_0$ be a
  polynomial with real coefficients such that for $n \in \Nat$ we have
  $f(n) \in \Nat$. Define a sequence by
  $x_n = 0.f(n)f(n+1)f(n+2)f(n+3)\dots$. By Lemma \ref{order lemma} we
  have $f^{-1}(n) = c_d^{-1/d}n^{1/d} + O(1)$. Thus, since $f(n)$ is increasing, we observe that
  for $J \in \Nat$ we have
  \begin{align*}
    A([0.1,0.2)& ; f^{-1}(2 \times 10^J); (x_n)_n) \\[1ex]
               &= (f^{-1}(2 \times 10^J) - f^{-1}(10^J)) + (f^{-1}(2 \times 10^{J-1}) - f^{-1}(10^{J-1})) + \dots \\[1ex]
               &\qquad+ (f^{-1}(20) - f^{-1}(10)) + O(J) \\[1ex]
               &= \sum_{i=1}^{J}(f^{-1}(2 \times 10^i) - f^{-1}(10^i)) + O(J) \\[1ex]
               &= \sum_{i=1}^{J}\left(\left(\frac{2 \times 10^i}{c_d}\right)^{1/d} + O(1) - \left(\frac{10^i}{c_d}\right)^{1/d} - O(1)\right) + O(J) \\[1ex]
               &= \frac{1}{c_d^{1/d}} \sum_{i=1}^{J}((2 \times 10^i)^{1/d} - (10^i)^{1/d}) + O(J) \\[1ex]
               &= \frac{(2^{1/d}-1)}{c_d^{1/d}} \sum_{i=1}^{J}{(10^{i/d})} + O(J),
  \end{align*}
  as required.
\end{proof}

The proof of Theorem \ref{Davenport-Erdos theorem} follows from Lemma
\ref{Davenport-Erdos A lemma} essentially as Theorem \ref{general kn
  theorem} follows from Lemma \ref{A estimate lemma} as we shall now
see.

\begin{proof}[Proof of Theorem \ref{Davenport-Erdos theorem}]
  In a similar fashion to the proof of Theorem \ref{general kn
    theorem} we will show that
  \[
    \limsup{J \to \infty}{\frac{A([0.1,0.2); f^{-1}(2 \times 10^J);
        (x_n)_n)}{f^{-1}(2 \times 10^J)}} > \frac{1}{9}.
  \]
  This would suffice to show that the sequence $(x_n)_n$ is not u.d.
  mod 1 over $[0.1,1)$ since it would show that
  \[
    \limsup{N \to \infty}{\frac{A([0.1,0.2); N; (x_n)_n)}{N}} >
    \frac{1}{9}.
  \]
  Now, by using Lemma \ref{Davenport-Erdos A lemma} and the formula for an infinite geometric series, for $J \in \Nat$ we
  have
  \begin{align*}
    \limsup{J \to \infty}{\frac{A([0.1,0.2); f^{-1}(2 \times 10^J); (x_n)_n)}{f^{-1}(2 \times 10^J)}} &= \limsup{J \to \infty}{\frac{\left(\frac{2^{1/d}-1}{c_d^{1/d}}\right) \sum_{i=1}^{J}{(10^{i/d})}+ O(1)}{\left(\frac{2 \times 10^J}{c_d}\right)^{1/d} + O(1)}} \\[1ex]
                                                                                                     &\geq
                                                                                                       \limsup{J
                                                                                                       \to
                                                                                                       \infty}{\frac{\left(\frac{2^{1/d}-1}{c_d^{1/d}}\right)
                                                                                                       \sum_{i=1}^{J}{(10^{i/d}})+
                                                                                                       O(1)}{2
                                                                                                       \times
                                                                                                       \left(\frac{2
                                                                                                       \times
                                                                                                       10^J}{c_d}\right)^{1/d}}}
    \\[1ex]
    &= \lim_{J \to \infty}{\frac{\left(\frac{2^{1/d}-1}{c_d^{1/d}}\right) \sum_{i=1}^{J}{(10^{i/d}})}{2 \times \left(\frac{2 \times 10^J}{c_d}\right)^{1/d}}} \\[1ex]
    &= \frac{(2^{1/d}-1)}{2 \times 2^{1/d}} \lim_{J \to \infty}{\sum_{i=0}^{J-1}{10^{\frac{-i}{d}}}} \\[1ex]
    &= \frac{5^{1/d}(2^{1/d}-1)}{2(10^{1/d}-1)}.
  \end{align*}
  Next, we will show that for any $n \in \Nat$ we have
  \[
    \frac{5^{1/n}(2^{1/n}-1)}{2(10^{1/n}-1)} \geq
    \frac{\log{2}}{\log{10}} > \frac{1}{9}.
  \]
  To save on notation, let us define
  $y_n = \frac{5^{1/n}(2^{1/n}-1)}{2(10^{1/n}-1)}$. We observe that
  $(y_n)_n$ is a monotonically decreasing sequence. Furthermore, by
  considering $\frac{5^{x}(2^{x}-1)}{2(10^{x}-1)}$ one may use
  l'H\^{o}pital's rule to show that
  $\lim_{x \to 0}{\frac{5^{x}(2^{x}-1)}{2(10^{x}-1)}} =
  \frac{\log{2}}{2\log{10}}$
  which, in turn, shows that
  $\lim_{n \to \infty}{y_n} = \frac{\log{2}}{2\log{10}}$.

  Since $(y_n)_n$ is a monotonically decreasing sequence it follows
  that $y_n \geq \frac{\log{2}}{2\log{10}}$ for all $n \in \Nat$. Thus,
  for any $d \in \Nat$,
  \[ \limsup{J \to \infty} \frac{A([0.1,0.2); f^{-1}(2 \times 10^J);
      (x_n)_n)}{f^{-1}(2 \times 10^J)} \geq
    \frac{5^{1/d}(2^{1/d}-1)}{2(10^{1/d}-1)} \geq \frac{\log 2 }{2\log
      10} > \frac{1}{9},\] as claimed.

  The proof of Theorem \ref{Davenport-Erdos theorem} is thus complete.
\end{proof}

\subsection{A comment on Benford's Law}
One may be thinking at this point that perhaps these results are a
consequence of Benford's Law \cite{Raimi1976} - after all, considering leading digits
(specifically the abundance of ones as a leading digit) is the crux of
the proofs of Theorems \ref{Champernowne number theorem},
\ref{general kn theorem} and \ref{Davenport-Erdos theorem}. We
shall conclude by taking a moment to discuss this.

For sequences, instead of Benford's Law itself, one considers the
notion of (strong) Benford sequences. Rather conveniently for us,
Cigler proposed a characterisation of Benford sequences in terms of uniform
distribution modulo 1 given below (see, for example,
\cite{Diaconisref}). For further discussion of this topic we refer the
readers to, for example, \cite{BNS2009}, \cite{Diaconisref} and
\cite{Raimi1976}.

\begin{prop}
  The sequence $(a_i)_i$ is a strong Benford sequence if and only if
  $(\log_{10}{a_i})_i$ is uniformly distributed modulo 1.
\end{prop}

It follows immediately from this characterisation that the sequence of natural
numbers is not a strong Benford sequence as $(\log_{10}{n})_n$ is not
u.d. mod $1$ (one possible way to show this is by modifying the argument of Example 2.4 in \cite[Chapter 1]{KUIPERS1974}). Also $ \log_b(f(n)) $ is not u.d.mod 1 for all bases $b$
and polynomials $f:\mathbb N \rightarrow \mathbb N$. To see this first
note that we can represent any polynomial as
$f(n) = (n - \alpha_1)(n- \alpha_2) \dots (n - \alpha_d)$ where $d$ is
the degree of the polynomial and the $\alpha_i$ are roots. Hence
$\log_b(f(n)) = \sum_{i = 1}^{d} \log_b(n-\alpha_i) \sim d \log_b(n)$
and so is not u.d. mod 1.


\vspace{0.5cm}
\textit{Acknowledgements.} \normalfont The authors would like to thank Christopher Hughes, Victor Beresnevich and Sanju Velani for providing useful comments on the manuscript. They are especially grateful to Chris for the many improvements to earlier drafts he has suggested.

DA and SB are supported by EPSRC Doctoral Training Grants EP/M506680/1 and EP/L505122/1, respectively.

\bibliographystyle{plain} \bibliography{}
\end{document}